\documentclass[11pt,reqno]{amsart}

\usepackage[utf8]{inputenc}
\usepackage[T1]{fontenc}
\usepackage{amsmath,amssymb,amsthm}
\usepackage{mathtools}
\usepackage[margin=1.1in]{geometry}
\usepackage{enumitem}
\usepackage{booktabs}
\usepackage{longtable}
\usepackage{array}
\usepackage{xcolor}
\usepackage[colorlinks=true,linkcolor=blue!50!black,citecolor=blue!50!black,urlcolor=blue!50!black]{hyperref}

\theoremstyle{plain}
\newtheorem{theorem}{Theorem}[section]
\newtheorem{lemma}[theorem]{Lemma}
\newtheorem{proposition}[theorem]{Proposition}
\newtheorem{corollary}[theorem]{Corollary}
\newtheorem{definition}[theorem]{Definition}
\newtheorem{remark}[theorem]{Remark}
\newtheorem{example}[theorem]{Example}

\newcommand{\Qp}{\mathbb{Q}_{>0}}
\newcommand{\Rp}{\mathbb{R}_{>0}}
\newcommand{\R}{\mathbb{R}}
\newcommand{\Z}{\mathbb{Z}}
\newcommand{\C}{\mathbb{C}}
\newcommand{\J}{J}

\newcommand{\Hom}{\operatorname{Hom}}
\newcommand{\Sol}{\operatorname{Sol}}
\newcommand{\supp}{\operatorname{supp}}
\newcommand{\arcosh}{\operatorname{arcosh}}

\begin{document}

\title[Reciprocal Cost on the Positive Rationals]{Reciprocal Cost on the Positive Rationals}

\author{Jonathan Washburn}
\address{Recognition Physics Institute, Austin, Texas, USA}
\email{jon@recognitionphysics.org}

\author{Sebastian Pardo-Guerra}
\address{Recognition Physics Institute, Austin, Texas, USA}
\email{
sebas@recognitionphysics.org}

\author{Milan Zlatanovi\'c}
\address{Department of Mathematics, Faculty of Science and Mathematics,
University of Ni\v{s}, Vi\v{s}egradska 33, 18000 Ni\v{s}, Serbia}
\email{zlatmilan@yahoo.com}

\date{}

\begin{abstract} We study nonnegative solutions of the reciprocal cost law on the
positive rationals and determine which of them admit regular extensions to the
positive reals. Using the substitution \(H=1+F\), we reduce the reciprocal composition law to
d'Alembert's functional equation. We prove that every nonnegative solution on
\(\Qp\) is determined by one real weight \(\alpha_p\) for each prime \(p\),
with only a global sign identification. Thus the rational solution space is
infinite dimensional. 
We prove that every nonnegative rational solution has an algebraic extension
to \(\Rp\),
but a regular extension exists exactly when the prime weights satisfy
\(\alpha_p=\lambda\log p\) for some \(\lambda\in\R\). In this case the
extension is unique and belongs to the one-parameter family
\(F_\lambda(x)=\cosh(\lambda\log x)-1\). Otherwise the
rational solution is unbounded on every nonempty open subset of \(\Qp\), and
the regular locus is closed and nowhere dense.   We also extend the result to arbitrary nontrivial subgroups of \(\Rp\), where
the alternative is governed by rational rank. Finally, we show that unit
logarithmic curvature selects the canonical reciprocal cost
\(\J(x)=(x+x^{-1})/2-1\), while, for a carrier \(G\) whose
logarithm is dense in \(\R\), the single asymptotic condition
\(F(e^s)\sim s^2/2\) implies both regularity and calibration.
 \\

\smallskip

\end{abstract}

\maketitle

\setcounter{tocdepth}{2}

\section{Introduction}

We study the following functional equation on positive ratios:
\begin{equation}\label{eq:rcl-intro}
F(xy)+F(x{y}^{-1})
=
2F(x)F(y)+2F(x)+2F(y).
\end{equation}  
On \(\Rp\), the composition law, continuity, and a unit quadratic calibration
together determine the canonical reciprocal cost (see e.g. \cite{CostUnique})
\[
 \J(x)=\frac{x+x^{-1}}2-1=\cosh(\log x)-1.
\]

There are two reasons to consider this equation on the positive rationals. First, for a continuous nonconstant \(F\) with \(F(1)=0\), every symmetric polynomial combiner of degree at most two has the form \(P(u,v)=2u+2v+cuv\), with \eqref{eq:rcl-intro} corresponding to \(c=2\) (see \cite{DAlembertInevitability}). Second, by unique factorization, \(\Qp\) has one independent coordinate
\(v_p(q)\) for each prime \(p\). Hence a rational solution has one free weight
\(\alpha_p\) for each prime, while a regular real extension exists exactly when these weights lie on the
one-dimensional logarithmic line
\[
\alpha_p=\lambda\log p,
\qquad p\in\mathcal P,
\]
for some \(\lambda\in\R\), where \(\mathcal P\) denotes the set of primes.
Thus the image of the restriction map
\[
\widetilde F\longmapsto \widetilde F|_{\Qp},
\]
defined on the regular nonnegative solutions on \(\Rp\), consists exactly of
the solutions \(F_\alpha\) with \(\alpha\) on this line. Every solution outside
this image is unbounded on every nonempty open subset of \(\Qp\).

{Before regularity and calibration, the nonnegative solution space on \(\Qp\)
is completely described by the prime weights. For
\(\alpha=(\alpha_p)_{p\in\mathcal P}\), set
\[
a_\alpha(q)=\sum_{p\in\mathcal P}v_p(q)\alpha_p .
\]
Then every nonnegative solution has the form
\[
F_\alpha(q)=\cosh(a_\alpha(q))-1,
\]
and \(F_\alpha=F_\beta\) if and only if \(\beta=\pm\alpha\). Hence, the solution space is
\(
\left(\prod_{p\in\mathcal P}\R\right)\big/\{\pm1\}.
\)}  {Every rational solution admits an algebraic extension to \(\Rp\) if no
regularity condition is required. A regular extension exists if and only if
\(
\alpha_p=\lambda\log p
\, (p\in\mathcal P)
\)
for some \(\lambda\in\R\). Here regular means continuous, locally bounded,
Haar-measurable, or monotone-profile. Otherwise, \(F_\alpha\) is unbounded
on every nonempty open subset of \(\Qp\). Hence the regular solutions form
the one-parameter family
\[
F_\lambda(x)=\cosh(\lambda\log x)-1.
\]
The unit quadratic calibration gives \(|\lambda|=1\). Since
\(F_\lambda=F_{-\lambda}\), both define the same reciprocal cost
\(\J\).}{
These results extend to arbitrary subgroups \(G\leq\Rp\). In this case,
regularity is determined by the rational rank of the logarithmic subgroup
\(\log G\). Moreover, the single asymptotic condition
\[
F(e^s)\sim \frac{s^2}{2}
\qquad (s\to0,\ s\in\log G)
\]
gives both regularity and the unit quadratic calibration.}

We state the three main results of the paper.  Their forms are given in
Theorems~\ref{thm:rational-moduli} and \ref{thm:extension},
Theorem~\ref{thm:rank}, and Theorem~\ref{thm:calibration-limit}.
\medskip

{\noindent\textbf{Theorem A.}
\emph{The nonnegative solutions on \(\Qp\) are
\(F_\alpha(q)=\cosh\bigl(\sum_p v_p(q)\alpha_p\bigr)-1\), and
\(F_\alpha=F_\beta\) exactly when \(\beta=\pm\alpha\).
Such a solution admits a regular extension if and only if there is a
\(\lambda\in\R\) such that \(\alpha_p=\lambda\log p\) for every prime \(p\).   
Equivalently, \(F_\alpha\) is continuous at some point of \(\Qp\), bounded
on some nonempty open subset of \(\Qp\), or has a nonnegative extension to
\(\Rp\) that is continuous, locally bounded, Haar-measurable, or
monotone-profile.
In this case the regular extension is unique and equals \(F_\lambda\).
Otherwise, \(F_\alpha\) is unbounded on every nonempty open subset of
\(\Qp\) and is continuous at no point.}}

\bigskip

\noindent\textbf{Theorem B.}
\emph{For a subgroup \(G\leq\Rp\), the same alternative is governed by the
rational rank \(r\) of \(\log G\): if \(r=1\) every nonnegative solution has a regular extension, and if \(r\geq2\) the regular locus is a nowhere-dense line.  Both
cases occur with \(G\) dense in \(\Rp\).}

\bigskip

\noindent\textbf{Theorem C.}
\emph{Let \(G\leq\Rp\) and suppose that \(\log G\) is dense in \(\R\).
If \(F:G\to[0,\infty)\) satisfies \eqref{eq:rcl-intro} and
\[
F(e^{s})\sim\frac{s^{2}}{2}
\qquad (s\to0,\ s\in\log G),
\]
then
\[
F(x)=\J(x)\qquad(x\in G).
\]
Thus the single asymptotic condition implies both regularity and unit quadratic
calibration.}
\medskip

%The description of the image of the restriction map is the main contribution
%of the paper. It separates three statements that are easy to conflate:
%\begin{enumerate}[label=\textup{(\roman*)}]
%\item algebraic comparison laws determine a character moduli space;
%\item regularity selects a one-dimensional subspace of this moduli space;
%\item unit calibration selects the reciprocal cost \(\J\).
%\end{enumerate}
%The earlier uniqueness theorem combines the last two steps. The present
%result identifies exactly which rational solutions lie in its domain of
%application; it does not change that theorem or any of its hypotheses.

We separate the classical ingredients from the new results. The regularity
part of Theorem A uses the classical fact that a locally bounded additive map
on a dense subgroup of \(\R\) is linear (see e.g. \cite{BinghamOstaszewski2018}).
The continuous classification on \(\Rp\) is also classical
\cite{Aczel1966,Kuczma2009}, and the family \(F_\lambda\), together with its
rescaling action, appears in \cite{CostUnique,AdmissibleCosts}.

The new results of this paper are the rational moduli, the exact image of the
regular restriction map, the blow-up outside this image, the role of rational
rank, and the single asymptotic calibration criterion of Theorem C.

Related Lean formalization is available in the Shape of Logic
repository\footnote{\url{https://github.com/jonwashburn/shape-of-logic}}.

\smallskip

\noindent{\bf Organization of the paper.}
Section~\ref{sec:algebra} reduces the reciprocal composition law to
d'Alembert's equation and proves the positivity and sign-rigidity lemmas.
Section~\ref{sec:rational} describes the rational moduli and the definite
locus. Section~\ref{sec:extension} proves the extension theorem, the blow-up
alternative, the closed nowhere-dense regular locus, and the rational-rank
criterion for arbitrary subgroups of \(\Rp\).
Section~\ref{sec:real} gives the real-side classification, the monotone
case, the automorphism action, and the calibration result.
Section~\ref{sec:sharpness} gives the counterexamples.
%%%%%%%%%%%%%%%%%%%%%%%%%%%%%%%%%%%%%%%%%%

\section{From reciprocal cost to d'Alembert}\label{sec:algebra}

Let \(G\) be an abelian group written multiplicatively, with identity element \(e\). We first introduce the reciprocal composition law.

\begin{definition}\label{def:rcl}
A function \(F:G\to\R\) satisfies the \emph{reciprocal composition law (RCL)} if
\begin{equation}\label{eq:rcl}
 F(xy)+F(xy^{-1})
 =2F(x)F(y)+2F(x)+2F(y)
 \qquad(x,y\in G).
\end{equation}
\end{definition}

{ \begin{lemma}\label{lem:identity-inverse}
Let \(F:G\to\R\) satisfy \eqref{eq:rcl}. Then \(F(e)\in\{0,-1\}\).
If \(F(e)=-1\), then \(F\equiv-1\). If \(F(e)=0\), then
\(
F(x^{-1})=F(x),\,x\in G.
\)
In particular, if \(F\geq0\), then \(F(e)=0\) and \(F\) is reciprocal.
\end{lemma} \begin{proof}
Putting \(x=y=e\) in \eqref{eq:rcl}, we have
\[
2F(e)=2F(e)^2+4F(e),
\]
and hence \(F(e)\bigl(F(e)+1\bigr)=0\). Therefore \(F(e)=0\) or \(F(e)=-1\).

Suppose first that \(F(e)=-1\). Setting \(y=e\) in \eqref{eq:rcl}, we obtain
\[
2F(x)=-2F(x)+2F(x)-2=-2,
\]
so \(F(x)=-1\) for every \(x\in G\).

Suppose now that \(F(e)=0\). Substituting \(x=e\) in \eqref{eq:rcl}, we get
\[
F(y)+F(y^{-1})=2F(y),
\]
and therefore \(F(y^{-1})=F(y)\).

Finally, if \(F\geq0\), the case \(F(e)=-1\) is impossible.
\end{proof}}
We say that \(F\) is \emph{normalized} if \(F(e)=0\), and \emph{reciprocal} if
\(F(x^{-1})=F(x)\) for every \(x\in G\). A solution is \emph{flat} if \(F\equiv0\). 

Let us define
\begin{equation}\label{eq:Hdef}
 H=1+F.
\end{equation}
With this substitution, the linear terms cancel.

\begin{lemma}\label{lem:affine}
A function \(F:G\to\R\) satisfies \eqref{eq:rcl} if and only if
\begin{equation}\label{eq:dalembert}
 H(xy)+H(xy^{-1})=2H(x)H(y),
\end{equation}
where \(H=1+F\).
\end{lemma}

\begin{proof}
Substituting \(F=H-1\) in \eqref{eq:rcl}, we obtain \eqref{eq:dalembert}.
\end{proof}

We use the following  result 
(see Kannappan~\cite{Kannappan1968} and
Stetk{\ae}r \cite{Stetkaer2013}).

\begin{theorem}\label{thm:dalembert-abelian}
Let \(G\) be an abelian group and let \(H:G\to\C\) satisfy
\eqref{eq:dalembert}, with \(H\not\equiv0\). Then there exists a homomorphism
\(\chi:G\to\C^\times\) (\(\C^\times=\C\setminus\{0\}\)) such that
\begin{equation}\label{eq:symchar}
 H(x)=\frac{\chi(x)+\chi(x)^{-1}}{2}
 \qquad (x\in G).
\end{equation}
\end{theorem}
In the normalized nonzero case one has \(H(e)=1\). No regularity assumption on
\(H\) is required.

\begin{lemma}\label{lem:positive-character}
Suppose \(H:G\to[1,\infty)\) is represented as in
\eqref{eq:symchar} by a homomorphism
\(\chi:G\to\C^\times\). Then
\(
\chi(G)\subset\Rp.
\)
\end{lemma}

\begin{proof}
Let us fix \(x\in G\) and write
\(\chi(x)=re^{i\theta}\), where \(r>0\).
Then
\[
H(x)=\frac{re^{i\theta}+r^{-1}e^{-i\theta}}{2}
\]
is real and satisfies \(H(x)\geq1\). Its imaginary part is
\[
\frac{r-r^{-1}}{2}\sin\theta.
\]

If \(r=1\), then \(H(x)=\cos\theta\geq1\), so
\(\cos\theta=1\) and \(\chi(x)=1\).

If \(r\neq1\), then \(\sin\theta=0\), so \(\chi(x)\) is real.
The negative possibility gives
\[
H(x)=\frac{-r-r^{-1}}{2}\leq-1,
\]
which is impossible. Hence \(\chi(x)>0\).
\end{proof}
{Writing \(a=\log\circ\chi\), every positive character
\(\chi:G\to\Rp\) gives an additive homomorphism
\(
a:G\to(\R,+).
\)
Conversely, every additive homomorphism \(a:G\to(\R,+)\) gives the
positive character \(\chi=e^a\). Therefore
\begin{equation}\label{eq:Fa}
 F_a(x)=\cosh(a(x))-1.
\end{equation}

{\begin{lemma}\label{lem:two-subgroups}
A group is not the union of two proper subgroups.
\end{lemma}

\begin{proof}
Let \(K=A\cup B\), where \(A\) and \(B\) are subgroups, and suppose that both
are proper. If \(A\subset B\), then \(K=B\), and if \(B\subset A\), then \(K=A\).
Both cases are impossible since \(A\) and \(B\) are proper subsets of \(K\).
Hence the 
sets $A\setminus B$ and \(B\setminus A\) are nonempty. Let us choose \(u\in A\setminus B\) and \(v\in B\setminus A\).
Then \(uv\) lies in neither \(A\) nor \(B\), which is a contradiction.
\end{proof}

The next lemma determines when two homomorphisms define the same function
\(F_a\).

\begin{lemma}\label{lem:sign-rigidity}
Let \(a,b:G\to(\R,+)\) be homomorphisms. If
\[
 \cosh(a(x))=\cosh(b(x))
 \qquad (x\in G),
\]
then \(b=a\) or \(b=-a\).
\end{lemma}

\begin{proof}
For real \(u,v\), the equality \(\cosh u=\cosh v\) holds if and only if
\(v=u\) or \(v=-u\). Hence the subgroups
\[
 A=\{x\in G:b(x)=a(x)\},
 \qquad
 B=\{x\in G:b(x)=-a(x)\}
\]
satisfy \(G=A\cup B\).
By Lemma~\ref{lem:two-subgroups}, \(A=G\) or \(B=G\), and hence
\(b=a\) or \(b=-a\).\end{proof}}

{We use the following elementary fact.

\begin{lemma}\label{lem:single-value}
Fix \(t_0>0\). The map
\(
\lambda\longmapsto\cosh(\lambda t_0)
\)
is strictly increasing on \([0,\infty)\). Consequently, if
\(\lambda,\mu\geq0\) and
\(
\cosh(\lambda t_0)=\cosh(\mu t_0),
\)
then \(\lambda=\mu\).
\end{lemma}

\begin{proof}For \(\lambda\geq0\) the argument \(\lambda t_0\) is nonnegative and strictly
increasing in \(\lambda\), and \(\cosh\) is strictly increasing on
\([0,\infty)\).
\end{proof}

Let \(\Sol_+(G)\) denote the set of nonnegative solutions of
\eqref{eq:rcl} on \(G\).

\begin{proposition}\label{prop:general-moduli}
Let \(G\) be an abelian group. The map
\[
a\longmapsto F_a,
\qquad
F_a(x)=\cosh(a(x))-1,
\]
is surjective from \(\Hom(G,\R)\) onto \(\Sol_+(G)\), and
\[
F_a=F_b
\quad\Longleftrightarrow\quad
b=a\ \text{or}\ b=-a.
\]
Therefore it induces a bijection
\begin{equation}\label{eq:general-moduli}
\Hom(G,\R)/\{\pm1\}
\ \longrightarrow\
\Sol_+(G).
\end{equation}
%The zero solution has the single representative \(a=0\), while every
%nonzero solution has the two representatives \(a\) and \(-a\).
\end{proposition}

\begin{proof}
Let \(F\in\Sol_+(G)\) and set \(H=1+F\). Then \(H\geq1\), so
\(H\not\equiv0\). By Lemma~\ref{lem:affine},
Theorem~\ref{thm:dalembert-abelian}, and
Lemma~\ref{lem:positive-character}, there is a positive character
\(\chi:G\to\Rp\) such that
\[
F(x)=\frac{\chi(x)+\chi(x)^{-1}}{2}-1.
\]
By substituting \(a=\log\circ\chi\), we obtain
\[
F(x)=\cosh(a(x))-1=F_a(x).
\]
Conversely, if \(a:G\to(\R,+)\) is a homomorphism, then the identity
\[
\cosh(u+v)+\cosh(u-v)=2\cosh u\cosh v
\]
shows that \(F_a\) satisfies \eqref{eq:rcl}. Finally,
Lemma~\ref{lem:sign-rigidity} gives
\[
F_a=F_b
\quad\Longleftrightarrow\quad
b=a\ \text{or}\ b=-a.
\]
\end{proof}}
{If we remove the condition \(F\geq0\), we have two families of real
solutions.

\begin{proposition}\label{prop:real-moduli}
Let \(G\) be an abelian group and let \(F:G\to\R\) satisfy
\eqref{eq:rcl}. Then either \(F\equiv-1\), or one of the following holds:
\(
F(x)=\varepsilon(x)\cosh(a(x))-1,
\)
where
\[
\varepsilon:G\to\{\pm1\},
\qquad
a:G\to(\R,+)
\]
are homomorphisms, or
\(
F(x)=\cos(\theta(x))-1,
\)
where
\(
\theta:G\to\R/2\pi\Z
\)
is a homomorphism.

The nonnegative solutions are
\[
F(x)=\cosh(a(x))-1,
\]
as in Proposition~\ref{prop:general-moduli}.

For \(G=\Qp\), the homomorphisms \(a\), \(\varepsilon\), and \(\theta\)
are determined by their values on the primes, with parameter sets
\[
\prod_p\R,\qquad
\prod_p\{\pm1\},\qquad
\prod_p(\R/2\pi\Z),
\]
respectively.
\end{proposition}

\begin{proof}
By Lemma~\ref{lem:identity-inverse}, either \(F\equiv-1\), or
\(F(e)=0\). In the second case, \(H=1+F\) satisfies \(H(e)=1\), and hence
\(H\not\equiv0\). Theorem~\ref{thm:dalembert-abelian} gives a
homomorphism \(\chi:G\to\C^\times\) such that
\[
H(x)=\frac{\chi(x)+\chi(x)^{-1}}{2}.
\]
For \(z\in\C^\times\), the number \(z+z^{-1}\) is real if and only if
\[
(z-\overline z)(1-|z|^{-2})=0.
\]
Thus \(z\in\R^\times\) or \(|z|=1\). Therefore the subgroups
\(
\chi^{-1}(\R^\times)\) and \(
\chi^{-1}(S^1)
\)
cover \(G\). By Lemma~\ref{lem:two-subgroups}, one of these subgroups is equal to \(G\). 
If \(\chi(G)\subset\R^\times\), write
\[
\varepsilon=\operatorname{sgn}\chi,
\qquad
a=\log|\chi|.
\]
Then \(\chi=\varepsilon e^a\), and
\[
H(x)=\varepsilon(x)\cosh(a(x)).
\]
If \(\chi(G)\subset S^1\), write \(\chi=e^{i\theta}\). Then
\[
H(x)=\cos(\theta(x)).
\]
Finally, \(F\geq0\) means \(H\geq1\). If
\(\chi(G)\subset\R^\times\), then \(H=\varepsilon\cosh a\) with
\(\cosh a\geq1>0\), so \(H\geq1\) implies \(\varepsilon\equiv1\). If
\(\chi(G)\subset S^1\), then \(H=\cos\theta\), so \(H\geq1\) implies
\(\cos\theta\equiv1\) and hence \(\theta\equiv0\).\end{proof}}

\section{The rational character moduli}\label{sec:rational}

{Every \(q\in\Qp\) has a unique factorization
\[
q=\prod_{p\in\mathcal P}p^{v_p(q)},
\]
where \(\mathcal P\) is the set of primes and
\(v_p(q)\in\Z\) is nonzero for only finitely many primes \(p\).
Therefore the map
\(
q\longmapsto \bigl(v_p(q)\bigr)_{p\in\mathcal P}
\)
gives an isomorphism
\begin{equation}\label{eq:q-direct-sum}
\Qp\cong\bigoplus_{p\in\mathcal P}\Z.
\end{equation}

An additive character \(a:\Qp\to\R\) is determined by the family
\(\alpha=(\alpha_p)_{p\in\mathcal P}\), where
\(\alpha_p=a(p)\). It is given by
\begin{equation}\label{eq:aq}
a_\alpha(q)=\sum_{p\in\mathcal P}v_p(q)\alpha_p.
\end{equation}
There is no finite-support condition on \(\alpha\), since the sum in
\eqref{eq:aq} is finite for every \(q\in\Qp\).}

{\begin{theorem}\label{thm:rational-moduli}
The nonnegative solutions of \eqref{eq:rcl} on \(\Qp\) are
\[
F_\alpha(q)
=
\cosh\!\left(
\sum_{p\in\mathcal P}v_p(q)\alpha_p
\right)-1,
\qquad
\alpha\in\prod_{p\in\mathcal P}\R.
\]
Moreover,
\[
F_\alpha=F_\beta
\quad\Longleftrightarrow\quad
\beta=\alpha\ \text{or}\ \beta=-\alpha.
\]
Consequently,
\begin{equation}\label{eq:rational-moduli}
\Sol_+(\Qp)
\cong
\left(
\prod_{p\in\mathcal P}\R
\right)\big/\{\pm1\}.
\end{equation}
\end{theorem}

\begin{proof}
We combine \eqref{eq:q-direct-sum}, \eqref{eq:aq}, and
Proposition~\ref{prop:general-moduli}.
\end{proof}

We call \(F_\alpha\) \emph{definite} if
\[
F_\alpha(q)=0 \quad\Longrightarrow\quad q=1.
\]

\begin{corollary}\label{cor:definite}
For \(\alpha\in\prod_{p\in\mathcal P}\R\), the following are equivalent:
\begin{enumerate}[label=\textup{(\roman*)}]
\item \(F_\alpha(q)=0\) implies \(q=1\);
\item \(a_\alpha:\Qp\to\R\) is injective;
\item the family \((\alpha_p)_{p\in\mathcal P}\) is linearly
independent over \(\Z\);%, meaning that every finite relation
%\[
%\sum_{p\in\mathcal P}n_p\alpha_p=0,
%\qquad n_p\in\Z,
%\]
%has \(n_p=0\) for every \(p\);
\item the family \((\alpha_p)_{p\in\mathcal P}\) is linearly
independent over \(\mathbb Q\).
\end{enumerate}
\end{corollary}

\begin{proof}
The equality \(F_\alpha(q)=0\) is equivalent to
\(a_\alpha(q)=0\). Therefore, \textup{(i)} and \textup{(ii)} are
equivalent. By unique factorization, a nontrivial element of the kernel of
\(a_\alpha\) is equivalent to a nontrivial finite integer relation among
the \(\alpha_p\). This proves the equivalence of \textup{(ii)} and
\textup{(iii)}. Finally, clearing denominators gives the equivalence of
\textup{(iii)} and \textup{(iv)}.
\end{proof}}

{ \begin{remark}\label{rem:definite-generic}
For \(\lambda\neq0\), the family
\(
(\lambda\log p)_{p\in\mathcal P}
\)
is linearly independent over \(\mathbb Q\). Hence every nonflat solution with
\[
\alpha_p=\lambda\log p
\qquad (p\in\mathcal P),
\]
and in particular \(\J\), is definite.

The indefinite locus in \(\prod_{p\in\mathcal P}\R\) is
\[
\bigcup_{n\neq0}
\left\{
\alpha\in\prod_{p\in\mathcal P}\R:
\sum_{p\in\mathcal P}n_p\alpha_p=0
\right\},
\]
where \(n=(n_p)\) ranges over the nonzero finitely supported integer
families. There are only countably many such families. Each set in this
union is closed and has empty interior: if \(n_p\neq0\), a sufficiently
small perturbation of the coordinate \(\alpha_p\) leaves the set while
remaining in the same basic open neighborhood. Therefore the indefinite
locus is meagre and the definite locus is comeagre. Since \(\prod_{p\in\mathcal P}\R\) is a Baire space, the definite locus is
also dense. By
Proposition~\ref{prop:thin}, the locus
\(
\alpha_p=\lambda\log p,
\, p\in\mathcal P,
\)
is nowhere dense.
\end{remark}}

  {\subsection{Single-prime sign changes}

The quotient by the global sign does not identify a sign change in a
single prime coordinate. Let us fix a prime \(p_0\), and define
\begin{equation}\label{eq:axis-twist}
\alpha^{(p_0)}_p=
\begin{cases}
-\log p_0, & p=p_0,\\
\log p, & p\neq p_0.
\end{cases}
\end{equation}
Let 
\[
F^{(p_0)}=F_{\alpha^{(p_0)}}.
\]

\begin{proposition}\label{prop:mixed-detection}
For every prime \(p\), we have
\(
F^{(p_0)}(p)=\J(p).
\)
If \(r\neq p_0\) is a prime, then
\[
F^{(p_0)}(p_0r)
=
\cosh(\log r-\log p_0)-1
\neq
\cosh(\log r+\log p_0)-1
=
\J(p_0r).
\]
\end{proposition}

\begin{proof}
The first equality follows from the evenness of \(\cosh\). Since
\(\log p_0>0\) and \(\log r>0\),
\[
|\log r-\log p_0|
<
\log r+\log p_0.
\]
The strict monotonicity of \(\cosh\) on \([0,\infty)\) gives the second
claim.
\end{proof}}

 {\section{The regular-extension boundary}\label{sec:extension}

Let
\begin{equation}\label{ll}
L=\log\Qp
=
\left\{
\sum_{p\in\mathcal P}n_p\log p:
n_p\in\Z,\ \supp(n)\ \text{finite}
\right\}
\subset\R.
\end{equation}
The family \((\log p)_{p\in\mathcal P}\) is linearly independent over
\(\mathbb Q\). After clearing denominators, any finite rational
relation gives
\[
\sum_{p\in\mathcal P}m_p\log p=0,
\qquad m_p\in\Z,
\]
and hence
\[
\prod_{p\in\mathcal P}p^{m_p}=1.
\]
By unique factorization, \(m_p=0\) for every \(p\).

Therefore every family
\[
\alpha=(\alpha_p)_{p\in\mathcal P}
\in\prod_{p\in\mathcal P}\R
\]
defines an additive map \(a_\alpha:L\to\R\) by
\[
a_\alpha\!\left(
\sum_{p\in\mathcal P}n_p\log p
\right)
=
\sum_{p\in\mathcal P}n_p\alpha_p.
\]

The further results use dense additive subgroups of \(\R\). The next lemma
shows that \(L\) has this property.

\begin{lemma}\label{lem:density}
Let \(r\) and \(s\) be distinct primes. Then
\(
\Z\log r+\Z\log s
\)
is dense in \(\R\). Consequently, \(L\) is dense in \(\R\), and the
subgroup of \(\Qp\) generated by \(r\) and \(s\) is dense in \(\Rp\).
\end{lemma}

\begin{proof}
The ratio \(\log r/\log s\) is irrational. Otherwise, there exist
positive integers \(m,n\) such that
\[
m\log r=n\log s,
\]
and hence \(r^m=s^n\), which is a contradiction to unique factorization.

Therefore the additive subgroup
\(
\Z\log r+\Z\log s
\)
is dense in \(\R\). Since
\(
L\supset \Z\log 2+\Z\log 3,
\)
the subgroup \(L\) is dense in \(\R\). The last statement follows from the fact
that \(\exp:\R\to\Rp\) is a homeomorphism.
\end{proof}}

{\subsection{Algebraic extensions}

\begin{proposition}\label{prop:algebraic-extension}
Every additive map \(a:L\to\R\) extends to an additive map
\(A:\R\to\R\). Consequently, every nonnegative rational solution
\(F_\alpha\) extends to a nonnegative solution of \eqref{eq:rcl} on
\(\Rp\).
\end{proposition}

\begin{proof}
First we extend \(a\) uniquely by rational linearity to
\(
V=\operatorname{span}_{\mathbb Q}L.
\)
Let us choose a Hamel basis of \(V\) and extend it to a Hamel basis of \(\R\)
over \(\mathbb Q\). We assign arbitrary real values to the additional basis
elements. The resulting \(\mathbb Q\)-linear map
\(
A:\R\to\R
\)
extends \(a\).

Let us define
\[
\widetilde F_A(x)=\cosh(A(\log x))-1.
\]
It is nonnegative, and by the addition formula for \(\cosh\) used in
Proposition~\ref{prop:general-moduli}, it satisfies \eqref{eq:rcl}. Its
restriction to \(\Qp\) is \(F_\alpha\).
\end{proof}

\begin{remark}
The extension from \(L\) to \(V=\operatorname{span}_{\mathbb Q}L\) is unique.
Choice is used only in extending from \(V\) to \(\R\) through a Hamel basis, and this extension is not unique. Any additive extension
\(A:\R\to\R\) not of the form
\(
A(t)=\lambda t
\)
is discontinuous and not Lebesgue measurable. Thus Proposition \ref{prop:algebraic-extension} gives
only algebraic extensions; no regularity is implied.
\end{remark}
}

 {\subsection{Regularity and the logarithmic line}

For functions on \(L\), continuity and local boundedness are used with
respect to the subspace topology inherited from \(\R\).

\begin{lemma}\label{lem:bounded-additive}
Let \(D\subset\R\) be a dense additive subgroup, and let
\(a:D\to\R\) be additive. The following are equivalent:
\begin{enumerate}[label=\textup{(\roman*)},leftmargin=2.2em]
\item \(a\) is continuous at \(0\);
\item \(a\) is bounded on a neighborhood of \(0\) in \(D\);
\item there is a unique \(\lambda\in\R\) such that \(a(t)=\lambda t\)
for every \(t\in D\).
\end{enumerate}
\end{lemma}

\begin{proof}
Condition \textup{(iii)} implies \textup{(i)} and \textup{(ii)}, and
\textup{(i)} implies \textup{(ii)}.

Assume \textup{(ii)}. Then there are \(M,\delta>0\) such that
\[
|a(t)|\leq M
\qquad
(t\in D,\ |t|<\delta).
\]
Let \(\varepsilon>0\), and choose an integer \(n>M/\varepsilon\). If
\(t\in D\) and \(|t|<\delta/n\), then \(|nt|<\delta\), and therefore
\[
n|a(t)|=|a(nt)|\leq M.
\]
Hence
\[
|a(t)|\leq\frac{M}{n}<\varepsilon.
\]
Thus \(a\) is continuous at \(0\). Since \(a\) is additive, it is then uniformly continuous on \(D\).

By density, \(a\) extends uniquely to a continuous additive map
\(A:\R\to\R\). Every continuous additive map on \(\R\) has the form
\(
A(t)=\lambda t
\)
for some \(\lambda\in\R\). Therefore \(a(t)=\lambda t\) for every
\(t\in D\). The parameter \(\lambda\) is unique because \(D\) contains
a nonzero element.
\end{proof}}
{
{\begin{lemma}\label{lem:cost-controls}
Let \(D\subset\R\) be an additive subgroup, let \(a:D\to\R\) be additive, and
define
\[
f(t)=\cosh(a(t))-1.
\]
Then \(f\) is continuous at \(0\) if and only if \(a\) is continuous at
\(0\). Moreover, \(f\) is bounded on a neighborhood of \(0\) if and only if
\(a\) is bounded on a neighborhood of \(0\).
\end{lemma}

\begin{proof}
The forward implications follow from the continuity of \(\cosh\).
Conversely,
\[
|a(t)|=\arcosh(f(t)+1).
\]
Since \(\arcosh\) is continuous and increasing on \([1,\infty)\),
continuity or local boundedness of \(f\) at \(0\) gives the corresponding
property for \(a\).
\end{proof}}}
{\begin{lemma}\label{lem:blowup}
Let \(D\subset\R\) be a dense additive subgroup, let \(a:D\to\R\) be
additive, and suppose that \(a\) is not of the form
\(
a(t)=\lambda t.
\)
Then
\[
f(t)=\cosh(a(t))-1
\]
is unbounded on every nonempty open subset of \(D\).
\end{lemma}

\begin{proof}
By Lemma~\ref{lem:bounded-additive}, \(a\) is unbounded in absolute value
on every neighborhood of \(0\).

Let \(U\subset D\) be nonempty and open, and choose \(t_0\in U\). There is
a neighborhood \(V\) of \(0\) in \(D\) such that
\[
t_0+V\subset U.
\]
Since
\[
a(t_0+v)=a(t_0)+a(v),
\]
the function \(a\) is unbounded in absolute value on \(t_0+V\). Since
\(\cosh x\to\infty\) as \(|x|\to\infty\), \(f\) is unbounded on
\(t_0+V\), and hence on \(U\).
\end{proof}}

\begin{corollary}\label{cor:local-regularity}
Let \(D\subset\R\) be a dense additive subgroup, let \(a:D\to\R\) be additive,
and let
\[
f(t)=\cosh(a(t))-1.
\]
Then the following are equivalent:
\begin{enumerate}[label=\textup{(\roman*)}]
\item \(f\) is continuous at some point of \(D\);
\item \(f\) is bounded on some nonempty open subset of \(D\);
\item there is a \(\lambda\in\R\) such that \(a(t)=\lambda t\) for every \(t\in D\).
\end{enumerate}
\end{corollary}

\begin{proof}
Condition \textup{(iii)} implies \textup{(i)} and \textup{(ii)}, while
\textup{(i)} implies \textup{(ii)} by continuity. Finally,
\textup{(ii)} implies \textup{(iii)} by Lemma~\ref{lem:blowup}.
\end{proof}

{The next lemma gives the measurable case. The monotone case is treated in
Theorem~\ref{thm:monotone-collapse}.

\begin{lemma}\label{lem:measurable-additive}
Let \(A:\R\to\R\) be additive. Suppose that \(A\) is bounded on a
Lebesgue measurable set of positive measure. Then there is a unique
\(\lambda\in\R\) such that
\(
A(t)=\lambda t
\, (t\in\R).
\)
In particular, the conclusion holds if \(|A|\) is Lebesgue measurable.
\end{lemma}

\begin{proof}
Let \(E\subset\R\) be measurable with positive measure, and suppose that
\[
|A(t)|\leq M
\qquad (t\in E).
\]
By Steinhaus's theorem~\cite{Steinhaus1920}, the set \(E-E\) contains an
interval \((-\delta,\delta)\). Thus, for every \(|t|<\delta\), there are
\(u,v\in E\) such that \(t=u-v\). Hence
\[
|A(t)|
=
|A(u)-A(v)|
\leq 2M.
\]
Therefore \(A\) is bounded on a neighborhood of \(0\). By
Lemma~\ref{lem:bounded-additive}, there is a unique \(\lambda\in\R\)
such that
\(
A(t)=\lambda t
,\, t\in\R.
\)

Suppose now that \(|A|\) is measurable. The sets
\[
E_n=\{t\in[0,1]:|A(t)|\leq n\},
\qquad n\in\mathbb N,
\]
are measurable and satisfy
\[
[0,1]=\bigcup_{n\in\mathbb N}E_n.
\]
Since \([0,1]\) has positive measure, at least one set \(E_n\) has
positive measure. On this set, \(|A|\leq n\), so the previous argument
gives
\(
A(t)=\lambda t
\, (t\in\R).
\)\end{proof}}

{\begin{definition}\label{def:regular}
Let \(G\leq\Rp\), let \(F\) be a nonnegative solution of
\eqref{eq:rcl} on \(G\), and let \(\widetilde F\) be a nonnegative
extension of \(F\) to \(\Rp\) satisfying \eqref{eq:rcl}. We call
\(\widetilde F\) \emph{regular} if it is continuous, bounded on a
neighborhood of \(1\), Haar measurable, bounded on a measurable set of
positive Haar measure, or if
\(
t\longmapsto 1+\widetilde F(e^t)
\)
is nondecreasing on \([0,\infty)\).
\end{definition}

{
\begin{theorem}\label{thm:extension}
Let
\(
\alpha=(\alpha_p)_{p\in\mathcal P}
\in\prod_{p\in\mathcal P}\R,
\)
and let \(F_\alpha\) be the corresponding nonnegative solution on
\(\Qp\). The following are equivalent:
\begin{enumerate}[label=\textup{(\roman*)}]
\item \(F_\alpha\) is continuous at \(1\) on \(\Qp\);
\item \(F_\alpha\) is bounded on a neighborhood of \(1\) in \(\Qp\);
\item there is a \(\lambda\in\R\) such that
\(
\alpha_p=\lambda\log p,\,p\in\mathcal P;
\)
\item \(F_\alpha\) has a continuous nonnegative extension to \(\Rp\)
satisfying \eqref{eq:rcl};
\item \(F_\alpha\) has a nonnegative extension to \(\Rp\) satisfying
\eqref{eq:rcl} that is bounded on a neighborhood of \(1\);
\item \(F_\alpha\) has a Haar-measurable nonnegative extension to
\(\Rp\) satisfying \eqref{eq:rcl};
\item \(F_\alpha\) has a nonnegative extension to \(\Rp\) satisfying
\eqref{eq:rcl} that is bounded on a measurable set of positive Haar
measure;
\item \(F_\alpha\) has an extension \(\widetilde F\) to \(\Rp\)
satisfying \eqref{eq:rcl} such that
\(
t\longmapsto 1+\widetilde F(e^t)
\)
is nondecreasing on \([0,\infty)\). 
\end{enumerate}

 In condition \textup{(viii)}, nonnegativity is not assumed. Since
\(\widetilde F(1)=F_\alpha(1)=0\), Lemma~\ref{lem:identity-inverse} implies that
\(\widetilde F\) is reciprocal. Hence
\(t\mapsto1+\widetilde F(e^t)\) is even, and its monotonicity on
\([0,\infty)\) implies nonnegativity.

When these conditions hold, the unique regular nonnegative extension is
\begin{equation}\label{eq:Flambda}
F_\lambda(x)
=
\cosh(\lambda\log x)-1
=
\frac{x^\lambda+x^{-\lambda}}{2}-1.
\end{equation}
The family \(\alpha\) determines \(\lambda\), while the solution determines
only \(|\lambda|\), since
\[
F_\lambda=F_{-\lambda}.
\]
If the equivalent conditions are not valid, then \(F_\alpha\) is unbounded on every
nonempty open subset of \(\Qp\).
\end{theorem}
}

By Corollary~\ref{cor:local-regularity} and the homeomorphism
\(\exp:L\to\Qp\), conditions \textup{(i)} and \textup{(ii)} are also
equivalent to continuity of \(F_\alpha\) at some point of \(\Qp\), and to
boundedness on some nonempty open subset of \(\Qp\).

\begin{proof}
We divide the proof into four steps.

\noindent \emph{Step 1.}
The map
\(
\exp:L\longrightarrow\Qp
\)
is a homeomorphism, where $L$ is given by \eqref{ll}. For \(t\in L\), we have
\[
F_\alpha(e^t)=\cosh(a_\alpha(t))-1.
\]
By Lemma~\ref{lem:density}, \(L\) is dense in \(\R\). Therefore
Lemmas~\ref{lem:cost-controls} and \ref{lem:bounded-additive} show that
\textup{(i)} and \textup{(ii)} hold if and only if
\(
a_\alpha(t)=\lambda t
\, (t\in L)
\)
for some \(\lambda\in\R\). Since
\[
a_\alpha(\log p)=\alpha_p,
\]
this is equivalent to
\[
\alpha_p=\lambda\log p
\qquad (p\in\mathcal P).
\]
Thus \textup{(i)}, \textup{(ii)}, and \textup{(iii)} are equivalent.
\smallskip

\noindent\emph{Step 2.}
Let us assume \textup{(iii)}. Define \(F_\lambda\) by
\eqref{eq:Flambda}. Then \(F_\lambda\) extends \(F_\alpha\). It is
continuous, Haar measurable, and bounded on every compact interval.
The function
\[
1+F_\lambda(e^t)=\cosh(\lambda t)
\]
is nondecreasing on \([0,\infty)\). Hence
\textup{(iv)}--\textup{(viii)} hold.
\smallskip

\noindent\emph{Step 3.}
Condition \textup{(iv)}, restricted to \(\Qp\), implies
\textup{(i)}. Condition \textup{(v)}, restricted to \(\Qp\), implies
\textup{(ii)}. Therefore both conditions imply \textup{(iii)} by
Step~1.

Let us assume \textup{(vi)}, and let \(\widetilde F\) be such an extension. Since \(\widetilde F\) is nonnegative and Haar
measurable,
\[
[1,2]
=
\bigcup_{n\in\mathbb N}
\{x\in[1,2]:\widetilde F(x)\leq n\}.
\]
Since \([1,2]\) has positive Haar measure, one of these measurable sets
has positive Haar measure. Hence \textup{(vii)} holds.

Let us assume \textup{(vii)}, and let \(\widetilde F\) be such an extension. By
Proposition~\ref{prop:general-moduli}, applied to \(\Rp\) and written
in logarithmic coordinates, there is an additive map
\(A:\R\to\R\) such that
\[
\widetilde F(e^t)=\cosh(A(t))-1.
\]
Suppose that
\(
\widetilde F(x)\leq M
\, (x\in E),
\)
where \(E\subset\Rp\) is measurable and has positive Haar measure. 
Since logarithmic coordinates identify multiplicative Haar measure
\(dx/x\) on \(\Rp\) with Lebesgue measure on \(\R\), the set
\(\log E\) is Lebesgue measurable and has positive measure. So, it holds
\[
|A(t)|
\leq
\arcosh(M+1)
\qquad (t\in\log E).
\]
By Lemma~\ref{lem:measurable-additive}, we have
\[
A(t)=\lambda t
\qquad (t\in\R).
\]
Restricting to \(L\), we obtain
\(
a_\alpha(t)=\lambda t,\, (t\in L),
\)
and hence \textup{(iii)}.

Finally, let us assume \textup{(viii)}. Since \(\widetilde F\) extends
\(F_\alpha\), we have
\[
\widetilde F(1)=F_\alpha(1)=0.
\]
Thus \(\widetilde F\) is normalized, and by
Lemma~\ref{lem:identity-inverse} it is reciprocal.

Let
\[
H(t)=1+\widetilde F(e^t).
\]
Reciprocity makes \(H\) even, and monotonicity on \([0,\infty)\) gives
\(
H(t)\geq1
,\,t\in\R.
\)
By Proposition~\ref{prop:general-moduli}, there is an additive map
\(A:\R\to\R\) such that
\[
H(t)=\cosh(A(t)).
\]
For \(0\leq t\leq1\), we have
\[
|A(t)|=\arcosh(H(t))
\leq\arcosh(H(1)).
\]
By evenness, the same holds for \(|t|\leq1\). Hence \(A\) is bounded
on a neighborhood of \(0\). Lemma~\ref{lem:bounded-additive} gives
\[
A(t)=\mu t
\qquad(t\in\R)
\]
for some \(\mu\in\R\). Restriction to \(L\) gives \textup{(iii)}. 
\smallskip

\noindent\emph{Step 4.}
Let \(\lambda\) satisfy \textup{(iii)}, and let \(\widetilde F\) be a
regular extension of \(F_\alpha\).

If \(\widetilde F\) is continuous, then it agrees with \(F_\lambda\) on
the dense subgroup \(\Qp\). Therefore
\[
\widetilde F=F_\lambda.
\]
Suppose that \(\widetilde F\) is bounded on a neighborhood of \(1\).
Write
\[
\widetilde F(e^t)=\cosh(A(t))-1
\]
with \(A:\R\to\R\) additive. By Lemma~\ref{lem:cost-controls}, \(A\)
is bounded on a neighborhood of \(0\). Lemma~\ref{lem:bounded-additive}
gives
\[
A(t)=\mu t
\qquad (t\in\R)
\]
for some \(\mu\in\R\). Hence
\(
\widetilde F=F_\mu.
\)

In the Haar-measurable case, condition \textup{(vi)} implies
\textup{(vii)}. In case \textup{(vii)},
Lemma~\ref{lem:measurable-additive} gives
\(
\widetilde F=F_\mu
\)
for some \(\mu\in\R\).

Since \(F_\mu\) and \(F_\lambda\) agree on \(\Qp\),
from Lemma~\ref{lem:sign-rigidity}, we get
\[
\mu=\lambda
\qquad\text{or}\qquad
\mu=-\lambda.
\]
Therefore
\[
F_\mu=F_\lambda.
\]
The monotone case follows similarly from
Theorem~\ref{thm:monotone-collapse}. Thus the regular extension is unique.

If the equivalent conditions are not valid, then \(a_\alpha\) is not linear on
\(L\). By Lemma~\ref{lem:blowup}, we have that
\[
t\longmapsto\cosh(a_\alpha(t))-1
\]
is unbounded on every nonempty open subset of \(L\). The homeomorphism
\(
\exp:L\longrightarrow\Qp
\)
then shows that \(F_\alpha\) is unbounded on every nonempty open subset
of \(\Qp\).
\end{proof}} 

\begin{remark}
If the equivalent conditions are not valid, then \(F_\alpha\) is continuous at no
point of \(\Qp\).
\end{remark}

%%%DOVDE

\begin{corollary}\label{cor:one-datum}
Let \(F:\Qp\to[0,\infty)\) satisfy \eqref{eq:rcl} and be continuous at
\(1\). Let \(q_0\in\Qp\), \(q_0\neq1\). If
\(
F(q_0)=\J(q_0),
\)
then
\[
F=\J
\qquad\text{on }\Qp.
\]
In particular, \(F(2)=1/4\) is sufficient.
\end{corollary}

\begin{proof}
By Theorem~\ref{thm:rational-moduli}, \(F=F_\alpha\) for some
\(\alpha\). Since \(F\) is continuous at \(1\),
by Theorem~\ref{thm:extension}, we have
\[
F=F_\lambda
\]
for some \(\lambda\in\R\). Since \(F_\lambda=F_{-\lambda}\), replacing \(\lambda\) by
\(|\lambda|\), we assume that \(\lambda\geq0\). Let us set
\[
t_0=|\log q_0|>0.
\]
The equality \(F(q_0)=\J(q_0)\) gives
\[
\cosh(\lambda t_0)=\cosh(t_0).
\]
From Lemma~\ref{lem:single-value}, we have \(\lambda=1\). Hence
\(
F=F_1=\J
\)
on \(\Qp\). Finally,
\[
\J(2)=\frac{2+2^{-1}}{2}-1=\frac14.
\]
\end{proof}}

{ \begin{corollary}\label{cor:finite-data}
Let \(S\subset\Qp\) be finite, and let \(F_\alpha\) be a nonnegative
solution. There are \(2^{\aleph_0}\) pairwise distinct nonnegative
solutions \(F_\beta\) such that
\[
F_\beta(q)=F_\alpha(q)
\qquad (q\in S),
\]
and every \(F_\beta\) is unbounded on every nonempty open subset of
\(\Qp\).
\end{corollary}

\begin{proof}
Only finitely many primes occur in the factorizations of the elements of \(S\).
Let us choose a prime \(r\) among the remaining ones, which is
possible because \(\mathcal P\) is infinite. For \(t\in\R\), define
\[
\beta^{(t)}_p=
\begin{cases}
t, & p=r,\\
\alpha_p, & p\neq r.
\end{cases}
\]
Then
\[
F_{\beta^{(t)}}(q)=F_\alpha(q)
\qquad (q\in S).
\]

At most one value of \(t\) gives
\(
\beta^{(t)}_p=\lambda\log p
, \,p\in\mathcal P
\)
for some \(\lambda\in\R\). After removing this value, there remain
\(2^{\aleph_0}\) parameters.  
If \((\alpha_p)_{p\neq r}\) is not identically zero, then
\(F_{\beta^{(t)}}=F_{\beta^{(t')}}\) implies \(t=t'\), since the global
sign alternative is impossible. If \(\alpha_p=0\) for every \(p\neq r\),
then \(F_{\beta^{(t)}}=F_{\beta^{(t')}}\) if and only if \(t'=\pm t\).
Restricting in this case to \(t>0\) still leaves \(2^{\aleph_0}\)
pairwise distinct solutions.
 By
Theorem~\ref{thm:extension}, every corresponding solution is unbounded
on every nonempty open subset of \(\Qp\).
\end{proof}

\begin{corollary}\label{cor:interleaving}
There is a sequence of nonnegative solutions \(F_n\) such that every
\(F_n\) is unbounded on every nonempty open subset of \(\Qp\), while
\[
F_n(q)\longrightarrow\J(q)
\qquad (q\in\Qp).
\]
\end{corollary}

\begin{proof}
Let
\[
\ell_p=\log p,
\qquad
\beta_2=1,
\qquad
\beta_p=0 \quad (p\neq2),
\]
and define
\[
\alpha^{(n)}=\ell+\frac1n\beta,
\qquad
F_n=F_{\alpha^{(n)}}.
\]
The family \(\alpha^{(n)}\) is not proportional to \(\ell\). Indeed, the
coordinates \(p\neq2\) give the proportionality factor \(1\), while
\[
\alpha^{(n)}_2=\log2+\frac1n\neq\log2.
\]
Hence, by Theorem~\ref{thm:extension}, each \(F_n\) is unbounded on every
nonempty open subset of \(\Qp\).

For fixed \(q\in\Qp\), we have
\[
a_{\alpha^{(n)}}(q)
=
\log q+\frac1n v_2(q)
\longrightarrow
\log q.
\]
Therefore
\[
F_n(q)
\longrightarrow
\cosh(\log q)-1
=
\J(q).
\]
\end{proof}}

\begin{corollary}\label{cor:dense-agreement}
For every prime \(p_0\), the cost \(F^{(p_0)}\) agrees with \(\J\) on a
dense subgroup of \(\Qp\), but is unbounded on every nonempty open subset
of \(\Qp\).
\end{corollary}

{\begin{proof}
Let \(K\leq\Qp\) be the subgroup generated by the primes \(p\neq p_0\).
By \eqref{eq:axis-twist}, we have
\[
F^{(p_0)}(q)=\J(q)\qquad (q\in K).
\]
Let us choose distinct primes \(r,s\neq p_0\). By Lemma~\ref{lem:density},
the subgroup generated by \(r\) and \(s\) is dense in \(\Rp\). Hence \(K\) is
dense in \(\Qp\).

The weights of \(F^{(p_0)}\) are not proportional to \((\log p)_{p\in\mathcal P}\).
Therefore Theorem~\ref{thm:extension} shows that \(F^{(p_0)}\) is unbounded on
every nonempty open subset of \(\Qp\).
\end{proof}}

\subsection{The regular locus}

Let us denote
\[
\mathcal A=\prod_{p\in\mathcal P}\R
\]
and consider it with the product topology. Let us also denote 
\(
\ell=(\log p)_{p\in\mathcal P}
\)
and define
\[
\mathcal D=\{\lambda\ell:\lambda\in\R\}\subset\mathcal A.
\]

\begin{proposition}\label{prop:thin}
The set \(\mathcal D\) is closed and nowhere dense in \(\mathcal A\).
Its image in \(\mathcal A/\{\pm1\}\) is also closed and nowhere dense.
After quotienting by sign, the regular locus is parametrized by
\(\lambda\geq0\).
\end{proposition}

\begin{proof}
Let us fix the prime \(2\). We have
\[
\mathcal D=
\bigcap_{p\in\mathcal P}
\left\{
\alpha\in\mathcal A:
\alpha_p\log 2=\alpha_2\log p
\right\}.
\]
Each set in this intersection is closed, so \(\mathcal D\) is closed.

{Let \(U\) be a nonempty basic open set in the product topology, and choose
\(\alpha\in U\). The set \(U\) restricts only finitely many coordinates.
Let us choose an unrestricted prime \(r\neq2\), and change the coordinate
\(\alpha_r\) so that
\[
\alpha_r\log 2\neq\alpha_2\log r.
\]
The resulting point belongs to \(U\), but not to \(\mathcal D\).} Thus
\(\mathcal D\) has empty interior and is nowhere dense.
The sign action preserves \(\mathcal D\). Let
\[
\pi:\mathcal A\longrightarrow\mathcal A/\{\pm1\}
\]
be the quotient map. Since the action is finite, \(\pi\) is closed.
Suppose that \(\pi(\mathcal D)\) contains a nonempty open set \(U\).
Then
\[
\pi^{-1}(U)\subset\mathcal D
\]
is nonempty and open, which is impossible. Therefore
\(\pi(\mathcal D)\) is closed and nowhere dense.

Finally, since \(F_\lambda=F_{-\lambda}\), the regular locus is
parametrized by \(\lambda\in[0,\infty)\).
\end{proof}

 %%%%%%%%%%%%%%%%%%%
\subsection{General carriers and rational rank}

Let \(G\leq\Rp\) be a nontrivial subgroup. Let us denote
\[
D=\log G,
\qquad
V=\operatorname{span}_{\mathbb Q}D,
\]
and let
\(
r=\dim_{\mathbb Q}V
\)
be the rational rank of \(G\).

By Proposition~\ref{prop:general-moduli}, we have
\[
\Sol_+(G)
\cong
\Hom_{\mathbb Q}(V,\R)\big/\{\pm1\}
\cong
\R^B\big/\{\pm1\},
\]
where \(B\) is a \(\mathbb Q\)-basis of \(V\).

Let us denote by
\(
\iota:V\longrightarrow\R
\) 
the inclusion map.

{
\begin{theorem}\label{thm:rank}
\begin{enumerate}[label=\textup{(\roman*)},leftmargin=2.2em]
\item if \(r=1\), every nonnegative solution on \(G\) is the restriction of
some \(F_\lambda\), and has a regular extension;
\item if \(r\geq2\), then \(D\) is dense in \(\R\), and a nonnegative solution
has a regular extension if and only if its corresponding additive character is \(\lambda\iota\).
Otherwise, the solution is unbounded on every nonempty open subset of \(G\).
The regular locus is the line \(\R\iota\), which is closed and nowhere dense
in \(\R^B\).
\end{enumerate}
\end{theorem}

\begin{proof}
Suppose first that \(r=1\). Then
\(
V=\mathbb Q u
\)
for some \(u\neq0\). Every \(\mathbb Q\)-linear map \(a:V\to\R\) has the form
\[
a(t)=\lambda t,
\qquad
\lambda=\frac{a(u)}{u}.
\]
This proves \textup{(i)}.

Suppose now that \(r\geq2\). Then \(D\) is not cyclic. Since every subgroup of
\(\R\) is either cyclic or dense, \(D\) is dense in \(\R\). The argument of
Theorem~\ref{thm:extension} applies with \(L\) replaced by \(D\), since the
regularity lemmas used there require only that \(D\) be a dense additive
subgroup of \(\R\). A regular extension exists exactly when
\[
a=\lambda\iota.
\]
Otherwise, Lemma~\ref{lem:blowup} shows that the solution is unbounded
on every nonempty open subset of \(G\). If \(B\) is infinite, the argument of Proposition~\ref{prop:thin} applies
directly. If \(B\) is finite, fix \(b_0\in B\). The locus \(\R\iota\) is
determined by 
\[
\alpha_b b_0=\alpha_{b_0}b
\qquad (b\in B).
\]
These relations are equivalent to
\(\alpha_b=\lambda b\) for all \(b\in B\), with
\(\lambda=\alpha_{b_0}/b_0\). Hence it is closed. Since \(r\geq2\), every nonempty basic open set contains a sufficiently small
perturbation of one coordinate which leaves this locus. Thus \(\R\iota\) has
empty interior and is nowhere dense.
\end{proof}
}
{
\begin{example}\label{ex:rank-one}
Both \(\Qp\) and \(\exp(\mathbb Q)\) are dense in \(\Rp\), but their rational
ranks are \(\aleph_0\) and \(1\), respectively. Therefore every nonnegative
solution on \(\exp(\mathbb Q)\) has a regular extension, while the regular
locus on \(\Qp\) is nowhere dense. Thus density alone does not determine
regularity.
\end{example}
}

\section{The real carrier}
\label{sec:real}

In this section we work on \(\Rp\). We first recover the one-parameter family
\(F_\lambda\) directly, by continuity and by monotonicity. We then study its
automorphisms and the calibration. \smallskip

{
We use the following classical result (see
Acz\'el~\cite{Aczel1966} and Kuczma~\cite{Kuczma2009}).

\begin{theorem}
\label{thm:continuous-trichotomy}
Let \(H:\R\to\R\) be continuous and satisfy
\[
H(t+s)+H(t-s)=2H(t)H(s),
\qquad
H(0)=1.
\]
Then exactly one of the following holds:
\[
H\equiv1,
\qquad
H(t)=\cosh(\lambda t)\quad(\lambda>0),
\qquad
H(t)=\cos(\tau t)\quad(\tau>0).
\]
In the last two cases the parameter is unique.
\end{theorem}

\begin{proposition}
\label{prop:continuous-collapse}
Let \(F:\Rp\to\R\) be normalized and continuous, and suppose that it satisfies
\eqref{eq:rcl}. Let us define
\(
G(t)=F(e^t).
\)
Suppose that
\[
\lim_{s\to0}\frac{G(s)}{s^2}=c
\]
for some \(c>0\). Then
\(
F=F_{\sqrt{2c}}.
\)
\end{proposition}

\begin{proof}
Let us define
\(
H=1+G.
\)
Then
\[
H(t+s)+H(t-s)=2H(t)H(s),
\qquad
H(0)=1.
\]
By Theorem~\ref{thm:continuous-trichotomy}, there are three possible
branches. For the flat branch, we have
\[
\lim_{s\to0}\frac{G(s)}{s^2}=0.
\]
For the hyperbolic branch,
\[
\lim_{s\to0}\frac{G(s)}{s^2}=\frac{\lambda^2}{2},
\]
while for the trigonometric branch,
\[
\lim_{s\to0}\frac{G(s)}{s^2}=-\frac{\tau^2}{2}.
\] Since \(c>0\), only the hyperbolic branch is possible.
Therefore
\(
F=F_\lambda
\)
for some \(\lambda>0\), and
\(
c=\frac{\lambda^2}{2}
\)
gives uniqueness.
\end{proof}}

The next theorem replaces continuity by an order condition, and it is  used in condition \textup{(viii)} of Theorem~\ref{thm:extension}.
 
{ 
\begin{theorem}\label{thm:monotone-collapse}
Let \(F:\Rp\to\R\) be normalized and satisfy \eqref{eq:rcl}. Let us define
\(
H_F(t)=1+F(e^t).
\)
If \(H_F\) is nondecreasing on \([0,\infty)\), then
\(
F=F_\lambda
\)
for some \(\lambda\in\R\). If \(F\) is nonflat, there is a unique
representative \(\lambda>0\).
\end{theorem}

\begin{proof}
By Lemma~\ref{lem:identity-inverse}, \(F\) is reciprocal. Therefore \(H_F\)
is even, and \(H_F(0)=1\). Since \(H_F\) is nondecreasing on
\([0,\infty)\), we have
\[
H_F(t)\geq1
\qquad(t\in\R).
\]
By Proposition~\ref{prop:general-moduli}, there is an additive map
\(A:\R\to\R\) such that
\[
H_F(t)=\cosh(A(t)).
\]
For \(0\leq t\leq1\),
\[
|A(t)|
=
\arcosh(H_F(t))
\leq
\arcosh(H_F(1)).
\]
Since \(H_F\) is even, the same bound holds for \(|t|\leq1\). Thus \(A\) is bounded
on a neighborhood of \(0\). By Lemma~\ref{lem:bounded-additive}, we have
\[
A(t)=\lambda t
\]
for some \(\lambda\in\R\). Hence
\[
F(x)=F_\lambda(x).
\]
Since \(F_\lambda=F_{-\lambda}\), a nonflat solution has a unique
representative \(\lambda>0\).
\end{proof}
}

 %%%%%%%
{
\subsection{Automorphisms and calibration}\label{sec:calibration}

For \(a\neq0\), let us define
\(
\sigma_a:\Rp\to\Rp,
\) by\[
\sigma_a(x)=x^a.
\]
Then
\begin{equation}\label{eq:gauge-action}
F_\lambda\circ\sigma_a
=
F_{a\lambda}
=
F_{|a|\lambda}.
\end{equation}

\begin{proposition}\label{prop:effective-action}
Every continuous multiplicative automorphism of \(\Rp\) has the form
\(\sigma_a\) for some \(a\in\R^\times\). On the family \(\{F_\lambda:\lambda\in\R\}\), the action factors through
\[
\R^\times/\{\pm1\}\cong\Rp.
\]
For every \(\lambda,\mu>0\), there is a unique class
\([a]\in\R^\times/\{\pm1\}\) such that
\[
F_\mu\circ\sigma_a=F_\lambda.
\]
\end{proposition}

\begin{proof}
Let us pass to logarithmic coordinates. Every continuous automorphism of
\((\R,+)\) has the form
\[
t\longmapsto at,
\qquad
a\neq0.
\]
Therefore every continuous multiplicative automorphism of \(\Rp\) has the form
\(
\sigma_a(x)=x^a.
\)

By \eqref{eq:gauge-action}, \(a\) and \(-a\) induce the same action on the family
\(\{F_\lambda:\lambda\in\R\}\), so the action factors through \(\R^\times/\{\pm1\}\).
Let \(\lambda,\mu>0\). By \eqref{eq:gauge-action}, we have
\[
F_\mu\circ\sigma_a=F_{|a|\mu}.
\]
By Lemma~\ref{lem:single-value}, we obtain
\[
F_\mu\circ\sigma_a=F_\lambda
\quad\Longleftrightarrow\quad
|a|\mu=\lambda.
\]
Therefore
\[
|a|=\frac{\lambda}{\mu}.
\]
Thus \(a=\pm\lambda/\mu\), and these two values define the same class in
\(\R^\times/\{\pm1\}\). Hence the class \([a]\) is unique.
\end{proof}
}

{\begin{theorem}
For \(F_\lambda\), the second derivative in logarithmic coordinates at the
identity is
\[
\left.\frac{d^2}{dt^2}F_\lambda(e^t)\right|_{t=0}
=
\lambda^2.
\]
Thus, within the family \(\{F_\lambda:\lambda\in\R\}\), the unit calibration
\[
\left.\frac{d^2}{dt^2}F_\lambda(e^t)\right|_{t=0}=1
\]
gives \(|\lambda|=1\), and hence \(F_\lambda=\J\).
\end{theorem}
\begin{proof}
We have
\[
F_\lambda(e^t)=\cosh(\lambda t)-1.
\]
Taking the second derivative at \(t=0\) gives
\[
\left.\frac{d^2}{dt^2}F_\lambda(e^t)\right|_{t=0}
=
\lambda^2.
\]
Thus unit calibration gives \(\lambda^2=1\). Since
\[
F_1=F_{-1}=\J,
\]
the proof is completed.
\end{proof}
}

\begin{theorem}\label{thm:calibration-limit}
Let \(G\leq\Rp\) be such that \(\log G\) is dense in \(\R\), and let
\(F:G\to[0,\infty)\) satisfy \eqref{eq:rcl}. If
\[
\lim_{s\to0,\ s\in\log G}
\frac{F(e^s)}{s^2}
=
\frac12,
\]
then
\(
F=\J
\)
on \(G\).
\end{theorem}

\begin{proof}
Let us denote
\(
D=\log G.
\)
By Proposition~\ref{prop:general-moduli}, there is an additive map
\(a:D\to\R\) such that
\[
F(e^s)=\cosh(a(s))-1.
\]
The assumption gives
\[
F(e^s)\longrightarrow 0
\qquad (s\to0,\ s\in D),
\]
so the map \(s\mapsto F(e^s)\) is bounded in a neighborhood of \(0\) in \(D\).
 By
Lemmas~\ref{lem:cost-controls} and \ref{lem:bounded-additive}, we have
\[
a(s)=\lambda s
\]
for some \(\lambda\in\R\). Therefore
\[
\lim_{s\to0,\ s\in D}
\frac{F(e^s)}{s^2}
=
\frac{\lambda^2}{2}.
\]
Hence \(\lambda^2=1\), and therefore
\[
F=F_1=\J.
\]
\end{proof}

In \cite{CostUnique}, the conclusion \(F=\J\) is obtained from continuity on
\(\Rp\) together with unit second derivative in logarithmic coordinates. Theorem~\ref{thm:calibration-limit}
shows that, when \(\log G\) is dense in \(\R\), the single condition
\[
F(e^s)\sim \frac{s^2}{2}
\qquad (s\to0,\ s\in\log G)
\]
 implies \(F=\J\) on \(G\).

\section{Sharpness}\label{sec:sharpness}
In this section, we give examples showing why the main assumptions are needed.

\begin{example}\label{ex:flat}
The function \(F\equiv0\) is nonnegative, normalized, continuous, monotone,
and satisfies \eqref{eq:rcl}. Its second derivative in logarithmic coordinates
at the identity is zero. Thus a positive second derivative is needed to obtain
a positive parameter.
\end{example}

\begin{example}\label{ex:cosine}
For \(\tau\in\R\), let us define
\[
F'_{\tau}(x)=\cos(\tau\log x)-1.
\]
Then \(F'_{\tau}\) is a continuous normalized solution of
\eqref{eq:rcl}. It takes negative values unless \(\tau=0\). Thus only continuity
assumption does not select the nonnegative hyperbolic family.
\end{example}

\begin{example}\label{ex:hamel}
Let \(A:\R\to\R\) be a non-linear additive map obtained from a Hamel basis.
Then
\[
F_A(x)=\cosh(A(\log x))-1
\]
is a nonnegative normalized solution of \eqref{eq:rcl}, but it is neither
continuous nor measurable. Indeed,
\[
|A(t)|=\arcosh\bigl(F_A(e^t)+1\bigr).
\]
If \(F_A\) is measurable, then \(|A|\) is measurable. In this case,
Lemma~\ref{lem:measurable-additive} gives
\(A(t)=\lambda t\), which is a contradiction.
 Thus algebraic extension does not imply regular
extension.
\end{example}

\begin{example}\label{ex:lambda-two}
Let us consider \(F_2=F_{\lambda=2}\), that is,
\[
F_2(x)=\frac{x^2+x^{-2}}2-1.
\]
This function is nonnegative, normalized, continuous, monotone in the
nonnegative logarithmic coordinate, and satisfies \eqref{eq:rcl}. However,
\[
F_2\neq\J,
\]
and its second derivative in logarithmic coordinates at the identity is \(4\).
Thus only regularity assumption does not determine \(\J\).
\end{example}
 
 \begin{example}
The cost \(F^{(p_0)}\) defined by \eqref{eq:axis-twist} agrees with \(\J\)
on every prime axis, but differs from \(\J\) on mixed products by
Proposition~\ref{prop:mixed-detection}. Thus agreement on the separate prime
axes does not determine the solution on \(\Qp\).
\end{example}
 
\section{Conclusion} 
 
 {We determined the nonnegative solutions of the reciprocal cost law on
\(\Qp\). Unique factorization gives one real parameter for each prime, and
the global sign is the only identification between parameter families. Thus
the full rational moduli space is
\[
\left(\prod_{p\in\mathcal P}\R\right)\big/\{\pm1\}.
\]

Every nonnegative rational solution has an algebraic extension to \(\Rp\),
but regular extension is more restrictive. It exists exactly when there is
\(\lambda\in\R\) such that
\[
\alpha_p=\lambda\log p
\qquad (p\in\mathcal P).
\]
In this case the regular nonnegative extension is unique and equals
\(F_\lambda\). Otherwise the rational solution is unbounded on every
nonempty open subset of \(\Qp\). The regular locus is closed and nowhere
dense in the rational moduli space.

We also proved the corresponding result for nontrivial subgroups of \(\Rp\).
%Rational rank determines the alternative.
In rank one every nonnegative
solution has a regular extension, while in rank at least two the regular
solutions form a one-dimensional closed nowhere dense locus.

Within the regular family, the unit second derivative in logarithmic
coordinates,
%\[
%\left.\frac{d^2}{dt^2}F(e^t)\right|_{t=0}=1,
%\]
gives \(|\lambda|=1\) and determines
\[
\J(x)=\frac{x+x^{-1}}2-1.
\]
When the logarithm of the carrier is dense in \(\R\), the single asymptotic condition \(F(e^s)\sim s^2/2\) as \(s\to0\) implies \(F=\J\) on the carrier.

Thus regularity selects the logarithmic line in the rational moduli space,
while calibration selects \(\J\):
\[
\left(\prod_{p\in\mathcal P}\R\right)\big/\{\pm1\}
\supset
\left\{
\big[(\lambda\log p)_{p\in\mathcal P}\big]:
\lambda\in\R
\right\}
\cong
\{F_\lambda:\lambda\geq0\}
\supset
\{\J\}.
\]
}

\bigskip

\vspace{6pt}
%%%%%%%%%%%%%%%%%%%%%%%%%%%%%%%%%%%%%%%%%%
\noindent {\bf Author Contributions:} {Conceptualization, J.W.; methodology, J.W., S.PG., and M.Z.; software, J.W.; validation, J.W., S.PG., and M.Z.; formal analysis, J.W., S.PG. and M.Z.; investigation, J.W., S.PG. and M.Z.; writing, original draft preparation, J.W.; writing, review and editing, S.PG. and M.Z.; funding acquisition, J.W. All authors have read and agreed to the published version of the~manuscript.}

\end{document}